\documentclass[a4paper,12pt]{article}
\usepackage{lmodern}
\usepackage[utf8]{inputenc}
\usepackage{geometry}
\usepackage{indentfirst}
\usepackage{caption}
\usepackage{amssymb,amsmath,amsthm}
\usepackage[hidelinks]{hyperref}
\usepackage{csquotes}
\usepackage{fancyhdr}
\usepackage{enumitem}
\usepackage{paracol}

\hypersetup{unicode = true,
        linktocpage = false,
         colorlinks = true,
         allcolors = red
}

\newcommand{\R}{{\mathbb{R}}}

\newcommand{\Q}{{\mathbb{Q}}}
\newcommand{\Co}{{\mathbb{C}}}
\newcommand{\N}{{\mathbb{N}}}
\newcommand{\Pri}{{\mathbb{P}}}

\newcommand{\Bor}[1]{\mathcal{B}\left(#1\right)}
\newcommand{\ul}[1]{\underline{#1}}
\newcommand{\meas}{\mathrm{meas}}

\newcommand{\gB}{\mathfrak{B}}

\DeclareMathOperator{\rank}{rank\!}

\newtheorem{theorem}{Theorem}

\theoremstyle{definition}
\newtheorem*{remark}{Remark}

\begin{document}

\title{{A modification of the mixed joint universality theorem \\ for a class of zeta-functions}}

\author{Benjaminas Togobickij, Roma Ka{\v c}inskait{\.e}}

\thispagestyle{empty}

\maketitle

\begin{abstract}
    The property of zeta-functions on mixed joint universality in the Voronin's sense states that any two holomorphic functions can be approximated simultaneously with accuracy $\varepsilon>0$ by suitable vertical shifts of the pair consisting from the Riemann zeta-  and Hurwitz zeta-functions. In \cite{RKKM2015}, it was shown rather general result, i.e., an approximating pair was composed of the Matsumoto zeta-functions' class and the periodic Hurwitz zeta-function.
    In this paper, we prove that this set of shifts has a strict positive density for all but at most countably many $\varepsilon>0$. 
    Also, we give the concluding remarks on certain more general mixed tuple of zeta-functions.


\end{abstract}

\section{Introduction}

Denote by  $\N$, $\N_0$, $\Pri$, $\Q$, $\R$, and  $\Co$ the sets of positive integers, non-negative integers, prime numbers, rational numbers, real numbers, and complex numbers, respectively, and by $s=\sigma+it$ a complex variable.


As it is well-known, the Riemann zeta-function $\zeta(s)$ is defined by the Dirichlet series and has a representation by the Euler product over primes, i.e.,
    $$
    \zeta(s)=\sum_{m=1}^{\infty}\frac{1}{m^s}=\prod_{p \in \Pri}\bigg(1-\frac{1}{p^s}\bigg), \quad \sigma>1.
    $$
    The Hurwitz zeta-function $\zeta(s,\alpha)$ with the real parameter $\alpha$, $0 < \alpha \leq 1$, is given by the series
    \[
    \zeta(s,\alpha)=\sum_{m=0}^{\infty}\frac{1}{(m+\alpha)^s}, \quad \sigma>1.
    \]
Both of the functions $\zeta(s)$ and $\zeta(s,\alpha)$ have an analytic continuation to the whole complex plane, except for a simple pole at $s=1$ with residue 1. Note that the function $\zeta(s,\alpha)$ has an Euler product over primes only for the cases $\alpha=1$ and  $\alpha=\frac{1}{2}$, when $\zeta(s,1)=\zeta(s)$ and $\zeta(s,\frac{1}{2})=(2^s-1)\zeta(s)$.

In the first decade of XXI century, the new type of universality was discovered by H.~Mishou \cite{HM2007} and independently by J.~Steuding and J.~Sander \cite{SJSJ2006}. They have opened so-called mixed joint universality in Voronin's sense for the Riemann zeta- and the Hurwitz zeta-functions. More precisely, they proved that a pair of analytic functions is simultaneously approximated by shifts of a pair $\big(\zeta(s+i\tau),\zeta(s+i\tau,\alpha)\big)$ with transcendental $\alpha$.

For the brevity, throughout this paper we use following notations and definitions. Let $D(a,b)=\{s\in \Co: a<\sigma<b\}$ for every  real numbers  $a<b$.
Denote by $\meas \{A\}$ the Lebesgue measure of a measurable set $A \subset \R$, by $\mathcal{B}(S)$ the set of all Borel subsets of a topological space $S$, and by $H(D)$ the set of all holomorphic functions on $D$.
For any compact set $K\subset\Co$, denote by $H^c(K)$ the set of all complex-valued continuous functions defined on $K$ and holomorphic in the interior of $K$, while  by $H^c_0(K)$ denote the subset of $H^c(K)$, consisting of all elements which are non-vanishing on $K$.

\begin{theorem}[see {\cite[Theorem~2]{HM2007}}]\label{te:Mish}
Suppose that $\alpha$ is a transcendental number. Let $K_1$ and $K_2$ be compact subsets of $D(\frac{1}{2},1)$ with connected complements. Then, for any $f_1(s)\in H^c_0(K_1)$, $f_2(s)\in H^c(K_2)$ and every $\varepsilon>0$, it holds that
\begin{equation}\label{eq:MishTeIneq}
            \liminf_{T\to\infty}\frac{1}{T}\meas\bigg\{\tau\in [0, T]: \sup_{s\in K_1}\left\vert\zeta(s+i\tau)-f_1(s)\right\vert<\varepsilon, \sup_{s\in K_2}\left\vert\zeta(s+i\tau,\alpha)-f_2(s)\right\vert<\varepsilon\bigg\}>0.
\end{equation}
\end{theorem}

In 2013, a statement of universality theorems in terms of density was proposed by J.-L.~Mauclaire \cite{JLM2013} and independently by  A.~Laurin{\v c}ikas and L.~Me{\v s}ka \cite{ALLM2014}. Such a statement of Theorem~\ref{te:Mish} was given in \cite{ALLM2016}.
%
%
To be precise, let
$$
L(\alpha,\Pri):=\big\{(\log p: p \in \Pri), (\log(m+\alpha): m \in \N_0)\big\}.
$$
Then the following modification of Theorem~\ref{te:Mish} was obtained.

\begin{theorem}[see {\cite[Theorem~2]{ALLM2016}}]\label{te:MJUM-LM}
Suppose that the elements of the set $L(\alpha,\Pri)$ are linearly independent over the field of rational numbers $\Q$. Let $K_1$, $K_2$, $f_1(s)$, $f_2(s)$ be as in Theorem~\ref{te:Mish}. Then the limit
\[
\lim_{T\to\infty}\frac{1}{T}\meas\left\{\tau\in [0, T]: \sup_{s\in K_1}\left\vert\zeta(s+i\tau)-f_1(s)\right\vert<\varepsilon, \sup_{s\in K_2}\left\vert\zeta(s+i\tau,\alpha)-f_2(s)\right\vert<\varepsilon\right\}>0
\]
exists for all but at most countably many $\varepsilon>0$.
\end{theorem}

More general results of the same type can be found, for example, in \cite{AL2019}.

The aim of this paper is to show an analogous result as of Theorem~\ref{te:MJUM-LM} for rather general classes of zeta-functions.

\section{Statement of new result}

A generalization of the Hurwitz zeta-function $\zeta(s,\alpha)$ was introduced by A.~Laurinčikas \cite{AL2006}.  For the periodic sequence $\gB=\{b_m \in \Co: m\in\N_0\}$ with a minimal period $k \in \N_0$ and a  fixed real  $\alpha$, $0<\alpha\leq 1$, the  periodic Hurwitz zeta-function $\zeta(s,\alpha;\gB)$ is defined by the Dirichlet series
\[
\zeta(s,\alpha;\gB)=\sum^{\infty}_{m=0}\frac{b_m}{(m+\alpha)^s}, \quad \sigma>1.
\]
Since the sequence $\gB$ is periodic sequence, then
$$
\zeta(s,\alpha;\gB)=\frac{1}{k^s}\sum_{m=0}^{k-1}b_m\zeta\bigg(s,\frac{m+\alpha}{k}\bigg).
$$
From this we deduce that the function $\zeta(s,\alpha;\gB)$ has an analytic continuation to the whole complex plane except for a simple pole at the point $s=1$ with residue $b=\frac{1}{k}\sum_{m=0}^{k-1}b_m$.

A class of the Matsumoto zeta-functions is a second class under our interest, particularly, since it covers a wide class of classical zeta-functions having the Euler product representation over primes. It was introduced by K.~Matsumoto in \cite{KM1990}. 
For every $m\in\N$, let $g(m)\in\N$, and, for $j \in \N$ with $1\leq j\leq g(m)$,  $f(j,m)\in\N$. Denote by $p_m$ the $m$th prime number, and $a^{(j)}_m\in\Co$. Assume that
\[\label{eq:MatsDzCondi}
g(m)\leq C_1p_m^{\alpha_0} \quad \text{and} \quad \left\vert a_m^{(j)}\right\vert\leq p_m^{\beta_0}
\]
with a positive constant $C_1$ and non-negative constants $\alpha_0$ and $\beta_0$. Define polynomials
\[
A_m(X)=\prod_{j=1}^{g(m)}\left(1-a_m^{(j)}X^{f(j,m)}\right)
\]
of degree $f(1,m)+\dots+f(g(m),m)$. The function
\begin{equation}\label{eq:MatsDzFun}
	\widetilde{\varphi}(s)=\prod_{m=1}^{\infty}\left(A_m(p_m^{-s})\right)^{-1}
\end{equation}
is called the Matsumoto zeta-function. The product on right-hand-side of the equality~\eqref{eq:MatsDzFun} converges absolutely for $\sigma>\alpha_0+\beta_0+1$. In this region, the function $\widetilde \varphi(s)$ has a Dirichlet series expansion as
\[
\widetilde{\varphi}(s)=\sum_{k=1}^{\infty}\frac{\widetilde{c}_k}{k^{s}}
\]
with the coefficients satisfying an estimate $\widetilde{c}_k=O(k^{\alpha_0+\beta_0+\varepsilon})$ for every  $\varepsilon>0$ if all prime factors of $k$ are large (for the details, see \cite{RKKM2017}). For brevity, we define its shifted version by
\[
\varphi(s):=\widetilde{\varphi}(s+\alpha_0+\beta_0)=\sum_{k=1}^{\infty}\frac{{c}_k}{k^{s}},
\]
where $c_k=\widetilde{c}_kk^{-\alpha_0-\beta_0}$. Then it is easy to see that  $\varphi(s)$ is absolutely convergent for $\sigma>1$.

Also we assume that the function $\varphi(s)$ satisfies the following conditions:
\begin{enumerate}
	\item[(i)]\label{lab:MatsCl1} $\varphi(s)$ can be meromorphically continued to the region $\sigma\geq\sigma_0$, $\frac{1}{2}\leq\sigma_0<1$, and all poles in this region belong to a compact set which has no intersections with the line $\sigma=\sigma_0$;
	\item[(ii)] $\varphi(\sigma+it)=O\left(\left\vert t\right\vert^{C_2}\right)$ for $\sigma\geq\sigma_0$ and  a positive constant $C_2$;
	\item[(iii)] it holds the mean-value estimate
	\[
	\int_0^T\left\vert\varphi(\sigma_0+it)\right\vert^2 d t=O(T).
	\]
\end{enumerate}
 All of functions satisfying above mentioned conditions construct the class of Matsumoto zeta-functions, and we denote the set of all such  $\varphi(s)$ as $\mathcal{M}$.

In 2015, R.~Ka{\v c}inskait{\. e} and K.~Matsumoto proved \cite{RKKM2015} a mixed joint universality theorem for wide class of Matsumoto zeta-functions and for the periodic Hurwitz zeta-function with transcendental parameter.

For the proof of mixed joint universality, the Bagchi method \cite{BB1981} can be used, but, in the case of the whole class $\mathcal{M}$ it is difficult to prove the denseness lemma. Therefore, we use one more restriction class, namely, the Steuding class ${\widetilde S}$ (see \cite{JS2007}).

We say that  $\varphi(s)$ belongs to the class $\widetilde{S}$ if the following assumptions are fulfilled:
\begin{enumerate}
	\item[(a)] $\varphi(s)$ has a Dirichlet series expansion $\varphi(s)=\sum_{m=1}^{\infty}a(m)m^{-s}$ with $a_m=O(m^\varepsilon)$ for every $\varepsilon>0$;
	\item[(b)] there exists $\sigma_{\varphi}<1$ such that $\varphi(s)$ can be meromorphically continued to the region $\sigma>\sigma_{\varphi}$, and is holomorphic there, except for a pole at $s=1$;
	\item[(c)] for any fixed $\sigma>\sigma_{\varphi}$ and any $\varepsilon>0$, there exists a constant $C_3\geq 0$ such that $\varphi(\sigma+it)=O(\left\vert t\right\vert^{C_3+\epsilon})$ ;
	\item[(d)] there exists the Euler product expansion
	\[
	\varphi(s) = \prod_{p\in\Pri}\prod_{j=1}^l\left(1-\frac{a_j(p)}{p^s}\right)^{-1};
	\]
	\item[(e)] there exists a constant $\kappa>0$ such that
	\[
	\lim_{x\to\infty}\frac{1}{\pi(x)}\sum_{p\leq x}\left\vert a(p)\right\vert^2=\kappa,
	\]
	where $\pi(x)$ denotes the number of primes $p$ up to $x$.
\end{enumerate}

Let $\varphi(s)\in {\widetilde{S}}$, and suppose that $\sigma^*$ is an infimum of all $\sigma_1$ for which
$$
\frac{1}{2T}\int_{-T}^{T}|\varphi(\sigma+it)|^2 d t \sim \sum_{m=1}^{\infty}\frac{|a(m)|^2}{m^{2\sigma}}
$$
holds for any $\sigma \geq \sigma_1$. Then $\frac{1}{2}\leq\sigma^*<1$, and we see that ${\widetilde S} \subset {\mathcal{M}}$.

In 2015, the first result on the mixed joint universality for the tuple  $\big(\varphi(s),\zeta(s,\alpha;\gB)\big)$ was proved by R.~Ka{\v c}inskait{\. e} and K.~Matsumoto (see \cite{RKKM2015}). Later it was proved in a more general situation extending collection of periodic Hurwitz zeta-functions (see \cite{RKKM2017}).

\begin{theorem}[{\cite[Theorem 2.2]{RKKM2015}}]\label{te:MixJoinUni}
	Suppose that $\varphi(s)\in\widetilde{S}$, and $\alpha$ is a transcendental number. Let $K_1$ be a compact subset of $D(\sigma^*,1)$, $K_2$ be a compact subset of $D\big(\frac{1}{2},1\big)$ and both with connected complements. Then, for any $f_1(s)\in H^c_0(K_1)$, $f_2(s)\in H^c(K_2)$ and $\varepsilon>0$, it holds that
	\[
	\liminf_{T\to\infty}\frac{1}{T}\meas\big\{\tau\in [0, T]: \sup_{s\in K_1}\left\vert\varphi(s+i\tau)-f_1(s)\right\vert<\varepsilon, \sup_{s\in K_2}\left\vert\zeta(s+i\tau,\alpha;\gB)-f_2(s)\right\vert<\varepsilon\big\}>0.
	\]
\end{theorem}

The aim of the present paper is to prove the modification of Theorem~\ref{te:MixJoinUni} in terms of density and  to give a further certain generalizations.

Now we state the main result of the present paper.

\begin{theorem}\label{te:MixJoinUniMod}
	Suppose that $\varphi(s) \in {\widetilde S}$ and $\alpha$ is a transcendental number. Let $K_1$, $K_2$, $f_1(s)$, $f_2(s)$ be
	as in Theorem~\ref{te:MixJoinUni}. Then, for all but at most countably many $\varepsilon>0$, it holds that
	\[
	\lim_{T\to\infty}\frac{1}{T}\meas\big\{\tau\in [0, T]:\sup_{s\in K_1}\left\vert\varphi(s+i\tau)-f_1(s)\right\vert<\varepsilon, \sup_{s\in K_2}\left\vert\zeta(s+i\tau,\alpha;\gB)-f_2(s)\right\vert<\varepsilon\big\}>0.
	\]
\end{theorem}

\begin{remark}
The transcendence of  $\alpha$ can be replaced by the assumption that the elements of the set $L(\alpha,\Pri)$ are linearly independent over $\Q$ as it is done in Theorem~\ref{te:MJUM-LM}.
\end{remark}

\section{Two probabilistic results}

For the proof of Theorem~\ref{te:MixJoinUniMod}, the probabilistic approach is used. In this section, we present joint mixed limit theorems on weakly convergent probability measures in the space of analytic functions and a proposition on the support of probability measure.

Since now, we are interested in the proof of a  joint limit theorem for the tuple $\big(\varphi(s), $ $\zeta(s,\alpha;\gB)\big)$, we deal with more specified regions than $D(\sigma^*,1)$ and $D(\frac{1}{2},1)$ (for the arguments, we refer to \cite{RKKM2015} or \cite{RKKM2017-Pal}).
As it is know,  the function $\varphi(s)$ has finitely many poles by condition (i) (denote them by $s_1(\varphi), \dots, s_l(\varphi)$), then put
\[
D_{\varphi}:=\{s \in \Co: \sigma>\sigma_0, \ \sigma \ne \Re s_j(\varphi), \  1\leq j\leq l\}.
\]
Since the function $\zeta(s,\alpha;\gB)$ can be written as a linear combination of the Hurwitz zeta-functions $\zeta(s,\alpha)$, it is entire or has at most a simple pole at $s=1$. Let
\[
D_{\zeta}:=
\begin{cases}
	\{s \in \Co: \sigma>\frac{1}{2}\}, & \text{if} \ \zeta(s,\alpha;\gB) \ \text{is entire},\\
	\{s \in \Co: \sigma>\frac{1}{2}, \ \sigma \ne 1\}, & \text{if} \ s=1 \ \text{is a pole of} \ \zeta(s,\alpha;\gB),
\end{cases}
\]
and $D_1$ and $D_2$ be two open regions of $D_\varphi$ and $D_\zeta$, respectively. By $H^2(D)$ we mean the Cartesian product of the spaces $H(D_1)$ and $H(D_2)$. Let $T>0$, and for $A \in \Bor{H^2(D)}$, define
\[
P_T(A)=\frac{1}{T}\meas\big\{\tau\in [0, T]: Z(\ul{s}+i\tau)\in A\big\}
\]
with $\ul{s}+i\tau=(s_1+i\tau,s_2+i\tau)$,  $s_1\in D_1$, $s_2\in D_2$, and
$$
Z(\ul{s})=\big(\varphi(s_1),\zeta(s_2,\alpha; \gB)\big).
$$

For the definition of limit measure, we need a certain probability space. Let ${\widehat \gamma}=\{s \in \Co : |s|=1 \}$. Define two tori
\[
\Omega_1=\prod_{p\in\Pri}\gamma_p \quad \text{and} \quad \Omega_2=\prod_{m=0}^{\infty}\gamma_m,
\]
where $\gamma_p={\widehat \gamma}$ for all $p \in \Pri$ and $\gamma_m={\widehat \gamma}$ for all $m\in\N_0$, respectively.
With the product topology and pointwise multiplication, both tori $\Omega_1$ and $\Omega_2$ become compact topological Abelian groups. Therefore, on $(\Omega_1,\Bor{\Omega_1})$ and $(\Omega_2,\Bor{\Omega_2})$, there exist the probability Haar measures $m_{1H}$ and $m_{2H}$ respectively. Thus, we get the probability spaces $(\Omega_1,\Bor{\Omega_1},m_{1H})$ and $(\Omega_2,\Bor{\Omega_2},m_{2H})$.  Denote by $\omega_1(p)$ the projection of $\omega_1\in\Omega_1$ to the coordinate space $\gamma_p$, $p \in \Pri$, while, for $m \in \N_0$, let $\omega_1(m):=\omega_1^{\alpha_{p_1}}(p_1) \cdots \omega_1^{\alpha_{p_r}}(p_r)$ according to the factorization of $m$ into prime divisors $m=p_1^{\alpha_1}\cdots p_r^{\alpha_r}$, and $\omega_2(m)$, $m \in \N_0$, the projection to the coordinate space $\gamma_m$.

Now let $\Omega=\Omega_1\times\Omega_2$, and denote the elements of $\Omega$ by $\omega=(\omega_1,\omega_2)$. Since $\Omega$ is a compact topological Abelian group, we can define the probability Haar measure $m_H:=m_{1H}\times m_{2H}$ on $(\Omega,\Bor{\Omega})$. This leads to a probability space $(\Omega, \Bor{\Omega},m_H)$.

On $(\Omega,\Bor{\Omega},m_{H})$, define $H^2(D)$-valued random element $Z(\ul{s}, \omega)$ by the formula
\[
Z(\ul{s}, \omega)=\big(\varphi(s_1,\omega_1),\zeta(s_2,\alpha,\omega_2;\gB)\big).
\]
Here $\ul{s}=(s_1,s_2)\in D_1\times D_2$,
\begin{align*}
\varphi(s_1,\omega_1)=\sum_{m=1}^{\infty}\frac{c_m\omega_1(m)}{m^{s_1}} \quad \text{and} \quad
\zeta(s_2,\alpha,\omega_2;\gB)=\sum_{m=0}^{\infty}\frac{b_m\omega_2(m)}{(m+\alpha)^{s_2}}
\end{align*}
are $H(D_1)$-valued and $H(D_2)$-valued random elements defined on $(\Omega_1,\Bor{\Omega_1},m_{1H})$ and $(\Omega_2,$ $\Bor{\Omega_2},m_{2H})$, respectively. Denote by $P_Z$ the distribution of the random element $Z(\ul{s},\omega)$, i.e.,
$$
P_Z(A)=m_H\big\{\omega\in\Omega: Z(\ul{s},\omega)\in A\big \}, \quad A \in \Bor{H^2(D)}.
$$

Now we are in position to state a mixed joint limit theorem for the tuple of the class of zeta-functions.

\begin{theorem}\label{te:JoinLim}
	Suppose that $\varphi(s) \in {\mathcal{M}}$, and $\alpha$ is a transcendental number. Then the measure $P_T(A)$ converges weakly to $P_{Z}(A)$ as $T\to\infty$.
\end{theorem}

\begin{proof}
	The proof of this theorem is given in \cite[Section 3, Theorem~2.1]{RKKM2015}. We only note that the transcendence of $\alpha$ plays an essential role in the proof.
\end{proof}

The second probabilistic result used in the proof of Theorem~\ref{te:MixJoinUniMod} is that we need to construct an explicit form of the support of the measure $P_Z$. For an obtaining of the mentioned result, we use the positive density method. Therefore, it is necessary to assume that the function $\varphi(s)$ belongs to the Steuding class ${\widetilde S}$, particularly, the condition (e) must to be satisfied (for the details, see \cite[Section~4, Remark~4.4]{RKKM2015}).

Let $\varphi(s)$, $K_1$, $K_2$, $f_1(s)$ and $f_2(s)$ be as in Theorem~\ref{te:MixJoinUni}. Then there exists a real number $\sigma_0$, $\sigma^*<\sigma_0<1$ and a sufficiently large positive number $M$ such that $K_1$ belongs to
$$
D_M=\{s \in \Co: \sigma_0<\sigma<1, \ |t|<M\}.
$$
Since $\varphi(s) \in {\widetilde S}$, it has only one pole at $s=1$, then put $D_\varphi=\{s \in \Co: \sigma>\sigma_0, \ \sigma \not = 1\}$. Therefore, $D_M \subset D_\varphi$. Analogously  we can find a sufficiently large positive number $N$ such that $K_2$ belongs to
$$
D_N=\bigg\{s \in \Co: \frac{1}{2}<\sigma<1, \ |t|<N \bigg \}.
$$

Now, if in Theorem~\ref{te:JoinLim} take $D_1=D_M$ and $D_2=D_N$, we get an explicit form of the $P_Z$'s support.

\begin{theorem}\label{lem:Supp}
	The support of the measure $P_{Z}$ is the set $S:=S_{\varphi}\times H(D_N)$, where
$S_{\varphi}:=\{f_1(s)\in H(D_M): f_1(s)\ne 0 \ \text{for all} \ s \in D_M, \ \text{or} \ f_1(s)\equiv 0\}$.
\end{theorem}

\begin{proof}
	The proof of the theorem can be found in \cite[Lemma 4.3]{RKKM2015}.
\end{proof}

\section{Proof of Theorem~\ref{te:MixJoinUniMod}}

First we remind two propositions used in the proof of the main result of the paper.

Recall  that a set $A\in\Bor{S}$ is said to be a continuity set of the probability  measure $P$ if $P(\partial A)=0$, where $\partial A$ is the boundary of $A$. Note that the set $\partial A$ is closed, therefore, it belongs to the class $\Bor{S}$. We are interested on the property of probability measures defined in terms of continuity sets, which is
equivalent to weak convergence. Therefore we use a following fact.

\begin{theorem}\label{te:ContSetEquiv}
	Let $P_n$ and $P$ be probability measures on $\left(S,\Bor{S}\right)$. Then the following assertions are equivalent:
	\begin{enumerate}
		\item[1)] $P_n$ converges weakly to $P$ as $n \to \infty$,
		\item[2)] $\lim\limits_{n\to\infty}P_n(A)=P(A)$ for all continuity sets $A$ of $P$.
	\end{enumerate}
\end{theorem}

\begin{proof}
	For the proof, see {\cite[Theorem 2.1]{PB1968}}.
\end{proof}

We also recall the Mergelyan theorem on the approximation of analytic functions by polynomials.

\begin{theorem}
Let $K \subset \Co$ be a compact subset with connected complement, and let $f(s)$ be a continuous function on $K$ analytic inside $K$.
Then, for any $\varepsilon>0$, there exists a polynomial $p(s)$
such that
$$
\sup_{s \in K}|f(s)-p(s)|<\varepsilon.
$$
\end{theorem}

\begin{proof}
The proof of the theorem can be found in \cite{SNM1952}.	
\end{proof}

\medskip

\begin{proof}[Proof of Theorem~\ref{te:MixJoinUniMod}]
Since $f_1(s)\not = 0$ on $K_1$, by the Mergelyan theorem, there exist polynomials $p_1(s)$ and $p_2(s)$ such that, for every $\varepsilon>0$,
\begin{align}\label{ineq:zetas}
	\sup_{s\in K_1}\big|f_1(s)-\exp(p_1(s))\big|<\frac{\varepsilon}{2}
	\quad \text{and} \quad
	\sup_{s\in K_2}\big|f_2(s)-p_2(s)\big|<\frac{\varepsilon}{2}.
\end{align}

In view of Theorem~\ref{lem:Supp}, an element $\big(\exp(p_1(s)),p_2(s)\big)$ belongs to the set $S$, i.e., to the support of the measure $P_Z$.

Consider the set
$$
G=\bigg\{
(g_1,g_2)\in H^2(D): \sup_{s\in K_1}\left\vert g_1(s)-\exp(p_1(s))\right\vert<\frac{\varepsilon}{2}, \ \sup_{s\in K_2}\left\vert g_2(s)-p_2(s)\right\vert<\frac{\varepsilon}{2}\bigg\}.
$$
This set is an open subset in $H^2(D)$ and, by Theorem~\ref{lem:Supp}, an open neigh\-bour\-hood of an element $\big(\exp({p_1(s)}),p_2(s)\big)$. Therefore, by Theorems~\ref{te:JoinLim} and \ref{te:ContSetEquiv}, the inequality $P_Z(G)>0$ holds.
	
Now, for $f_1(s)$ and $f_2(s)$ fulfilling the conditions of Theorem~\ref{te:MixJoinUniMod},  define the set $G_\varepsilon$ by
\[
G_{{\varepsilon}}=\bigg\{(g_1,g_2)\in H^2(D): \sup_{s\in K_1}|g_1(s)-f_1(s)|<\frac{\varepsilon}{2}, \ \sup_{s\in K_1}|g_2(s)-f_2(s)|<\frac{\varepsilon}{2} \bigg \}
\]
with the boundary
\begin{align*}
		\partial G_{\epsilon}=&\big\{(g_1,g_2)\in H^2(D): \sup_{s\in K_1}|g_1(s)-f_1(s)|<\varepsilon, \ \sup_{s\in K_2}|g_2(s)-f_2(s)|=\varepsilon\big \} \\
		&\cup\big\{(g_1,g_2)\in H^2(D): \sup_{s\in K_1}|g_1(s)-f_1(s)|=\varepsilon, \ \sup_{s\in K_2}|g_2(s)-f_2(s)|<\varepsilon\big \} \\
		&\cup\big\{(g_1,g_2)\in H^2(D): \sup_{s\in K_1}|g_1(s)-f_1(s)|=\varepsilon, \ \sup_{s\in K_2}|g_2(s)-f_2(s)|=\varepsilon\big \}.
\end{align*}
Easy to see that with different $\varepsilon_1>0$ and  $\varepsilon_2>0$ the boundaries $\partial G_{\varepsilon_1}$ and $\partial G_{\varepsilon_2}$ are disjoint. Therefore, only countable many sets $\partial G_\varepsilon$ can have positive measure $P_Z$. Hence, $P_Z(\partial G_\varepsilon)=0$ for at most countable set of values $\varepsilon>0$, i.e., $G_\varepsilon$ is a continuity set of $P_Z$ for all but at most countable many $\varepsilon>0$. Moreover, in view of \eqref{ineq:zetas}, $G \subset G_\varepsilon$. Therefore, by Theorem~\ref{te:JoinLim}, we have
$$
\lim_{T\to \infty}P_T({G_\varepsilon})=P_Z(G_{\varepsilon})>0
$$
for all but at most countable many $\varepsilon>0$. This together with the definitions of $P_T$ and $G_\varepsilon$ prove the theorem.
\end{proof}

\section{Concluding remarks}

Theorem~\ref{te:MixJoinUniMod} can be generalised in a following direction.

Suppose that $\alpha_j$ is a real number such that $0<\alpha_j<1$, and $l(j)$ is a positive integer, $j=1, \dots, r$. Let $\lambda=l(1)+\dots+l(r)$. For each $j$ and $l$, $1\leq j\leq r$, $1\leq l\leq l(j)$, let $\gB_{jl}=\{b_{mjl}:m\in\N_0\}$ be a periodic sequence of complex numbers $b_{mjl}$ with minimal period $k_{jl}$, and let $\zeta(s,\alpha_j;\gB_{jl})$ be the corresponding periodic Hurwitz zeta-function. Denote by $k_j$ the least common multiple of periods $k_{j1}, \dots, k_{jl(j)}$. Let $B_j$ be a matrix consisting of elements $b_{mjl}$ from the periodic sequences $\gB_{jl}$, $j=1,\dots,r$, $l=1,\dots,l(j)$, i.e.,
\[
B_j : = \begin{pmatrix}
	b_{1j1} & \cdots & b_{1jl(j)} \\
	\vdots & \ddots & \vdots \\
	b_{k_jj1} & \cdots & b_{k_jjl(j)}
\end{pmatrix}, \quad j=1,\dots,r.
\]

\begin{theorem}\label{te:apMixJoinUniMod}
Suppose that $\alpha_1,\dots,\alpha_r$ are algebraically independent over $\Q$, $\rank\left(B_j\right)=l(j)$, $1\leq j\leq r$, and $\varphi(s)$ belongs to the class $\widetilde{S}$. Let $K_1$ be a compact subset of $D(\sigma^*,1)$ and $K_{2jl}$ be a compact subset of $D(\frac{1}{2},1)$, all of them with connected complements. Suppose that $f_1(s)\in H^c_0(K_1)$ and $f_{2jl}(s)\in H^c(K_{2jl})$. Then, for all but at most countably many $\varepsilon>0$, it holds that
\begin{align}\label{mod:UnivIneq}
		\lim_{T\to\infty}\frac{1}{T}\meas
		\bigg\{\tau\in [0, T]: &
		\sup_{s\in K_1} |\varphi(s+i\tau)-f_1(s)|<\varepsilon,\cr &\sup_{1\leq j\leq r}\sup_{1\leq l\leq l(j)}\sup_{s\in K_{2jl}}|\zeta(s+i\tau,\alpha_j;\gB_{jl})-f_{2jl}(s)|<\varepsilon
		\bigg\}>0.
\end{align}
\end{theorem}

\begin{proof}
	In \cite{RKKM2017}, the joint mixed universality it was proved  under the same conditions as in the theorem instead $\lim$ studying $\liminf$ for every $\varepsilon>0$. Therefore, arguing in similar way as in the proof of Theorem~\ref{te:MixJoinUniMod}, we can show the universality inequality~\eqref{mod:UnivIneq}.
	
	However, we give some highlights. Let $H^{\lambda+1}(D):=H(D_1)\times\underbrace{H(D_2)\times H(D_2)}_\lambda$, and let $p_{2jl}(s)$ be a polynomials satisfying the second inequality of \eqref{ineq:zetas} for each $j=1,\dots,r$, $l=1,...,l(j)$.
	
	Instead of the set $G$ in the proof of Theorem~\ref{te:MixJoinUniMod}, we consider the set
\begin{align*}
	{\underline G}=\bigg\{&
	(g_1, g_{211},\dots, g_{21l(1)},\dots, g_{2r1},\dots, g_{2rl(r)})\in H^{\lambda+1}(D): \\
	&\quad \sup_{s\in K_1}|g_1(s)-\exp(p_1(s))|<\frac{\varepsilon}{2}, \ \sup_{1\leq j\leq r}\sup_{1\leq l\leq l(j)}\sup_{s\in K_{2jl}}|g_{2jl}(s)-p_{2jl}(s)|<\frac{\varepsilon}{2}\bigg\},
\end{align*}
and show that $P_{\underline Z}(\underline{G})>0$, where $P_{\underline Z}$ is a distribution of the $H^{\lambda+1}(D)$-random element constructed for the collections of zeta-functions in the theorem. For details, we refer to \cite{RKKM2017}.

Next we define the set
\begin{align*}
	{\underline G_\varepsilon}=\bigg\{&
	(g_1, g_{211},\dots, g_{21l(1)},\dots, g_{2r1},\dots, g_{2rl(r)})\in H^{\lambda+1}(D): \\
	&\quad \sup_{s\in K_1}|g_1(s)-f_1(s)|<{\varepsilon}, \ \sup_{1\leq j\leq r}\sup_{1\leq l\leq l(j)}\sup_{s\in K_{2jl}}|g_{2jl}(s)-f_{2jl}(s)|<{\varepsilon}\bigg\}
\end{align*}
and obtain that it is a continuity set of the measure $P_{\underline Z}$ for all but at most countably many $\varepsilon>0$. Again, arguing as for $G$ and $G_\varepsilon$, we get that ${\underline G}\subset {\underline G}_\varepsilon$. Therefore, for  all but at most countably many $\varepsilon>0$, $\lim\limits_{T \to \infty} P^*_T({\underline G}_\varepsilon)=P_{\underline Z}({\underline G}_\varepsilon)>0$. In view of the similarity of $P^*_T$'s construction to $P_T$ extending a collection of the periodic Hurwitz zeta-functions (for exact definition of $P^*_T$, see p.~195 in  \cite{RKKM2017}), this and the definition of ${\underline G}_\varepsilon$  complete the proof.
\end{proof}

Finally, we would like to mention that Theorem~\ref{te:apMixJoinUniMod} can be shown under different conditions than  the algebraic independence over $\Q$ of the parameters the parameters $\alpha_1,\dots,\alpha_r$. Particularly, we can prove that the universality inequality \eqref{mod:UnivIneq} holds if the elements of the set $\big\{(\log p: p \in \Pri), (\log(m+\alpha_j): m \in \N_0, j=1,...,r)\big\}$ are linearly independent over $\Q$.


Address of the auhtors:\\
Department of Mathematics and Statistics, \\
Faculty of Informatics, \\
Vytautas Magnus University, \\
Vileikos 8, Kaunas LT-44404, Lithuania\\ \\
e-mails: roma.kacinskaite@vdu.lt, benjaminas.togobickij@vdu.lt

\end{document}